\documentclass[letterpaper,10pt]{article}
\usepackage{amsfonts,amssymb,amsthm,eucal,amsmath}
\usepackage{graphicx}
\usepackage[T1]{fontenc}
\usepackage{latexsym,url}

\newtheorem{thm}{Theorem}[section]

\newtheorem{cor}[thm]{Corollary}
\theoremstyle{definition}
\newtheorem{defn}[thm]{Definition}

\newtheorem{rmk}[thm]{Remark}

\newcommand{\til}[1]{\widetilde{#1}}

\newcommand{\bdry}{\partial}

\title{On spun-normal and twisted squares surfaces}
\author{Henry Segerman}

\begin{document}

\maketitle

\section{Introduction}

Yoshida in \cite{yoshida91} and Tillmann in \cite{tillmann_degenerations} describe different methods of producing surfaces within a 3-manifold $M$ with boundary from ideal points of the deformation variety of $M$, given a particular ideal tetrahedralisation of $M$.  The Yoshida construction builds a surface from twisted squares within the tetrahedra, the Tillmann construction results in a spun-normal surface relative to the tetrahedra. In this paper we investigate the connection between the methods. We show first that a twisted squares surface can be obtained from the spun-normal surface by an isotopy and possibly, two different types of compression moves (the usual disk compression moves, and compressions along certain annuli).\\

Both constructions, with a mild condition on the ideal point we start from, result in surfaces which although they may not be essential, can be reduced to essential surfaces by (disk) compressions and possibly deleting components\footnote{
There are two different senses of an "essential" surface in the literature which differ in whether or not they require the surfaces to be boundary incompressible or not. Yoshida does require boundary incompressibility (and refers to the surfaces as "incompressible and boundary incompressible") whereas Tillmann does not. We will use the version that does not require boundary incompressibility in this paper, although all of the arguments also go through with the other version.
}. Tillmann shows that if the ideal point of the deformation variety corresponds to an ideal point of the character variety then the corresponding spun-normal surface is dual to the ideal point of the character variety in the sense of Shalen~\cite{handbook_shalen}. We use Tillmann's result to show that the same is true for twisted squares surfaces produced by Yoshida's construction, and so any essential surface obtained from either the spun-normal or twisted squares construction in this case is also detected by the character variety. We use this in \cite{segerman_torus_bundles} to show that all non fiber and non semi-fiber incompressible surfaces in punctured torus bundles are detected by the character variety.\\

The author thanks Cameron Gordon and Stephan Tillmann for helpful discussions.

\subsection{The projective admissible solution space and spun-normal form}
Spun-normal surfaces are made up of two types of surface parts which sit within the tetrahedra of the ideal tetrahedralisation: triangles and quadrilaterals. The data that Tillmann uses to produce a spun-normal surface is the number and positioning of quadrilaterals in each tetrahedron, subject to the "Q-matching equations" (see Tillmann~\cite{tillmann_norm_surf}). The projective admissible solution space is the set of solutions to the Q-matching equations such that at most only one type (positioning) of quadrilateral appears in each tetrahedron (this makes the solution "admissible") and normalised so that the sum of the coordinates is one. Given such a solution with rational coordinate ratios, the solution is scaled again so that the coordinates are integers, with greatest common divisor of one, and these give us the number and type of quadrilaterals to use.\\

Theorem 2.4 of \cite{tillmann_norm_surf} states that an admissible integer solution of the Q-matching equations corresponds to a spun-normal surface. The surface is constructed by inserting the appropriate number of quadrilateral pieces and infinitely many triangle pieces in each tetrahedron, making sure that the boundaries of those pieces can be made to match up by isotoping them within each tetrahedron, then removing any boundary parallel components.

\subsection{The Yoshida construction}

The Yoshida construction in \cite{yoshida91} uses the same data (an admissible integer solution to the Q-matching equations, although Yoshida does not use this terminology) but in a different way to produce what we call a twisted squares surface. In place of a number of parallel quadrilaterals within a given tetrahedron, we put the same number of parallel twisted squares, see figure \ref{fig_quad_and_tw_sq.pdf}. \\

\begin{figure}[htbp]
\centering
\includegraphics[width=0.6\textwidth]{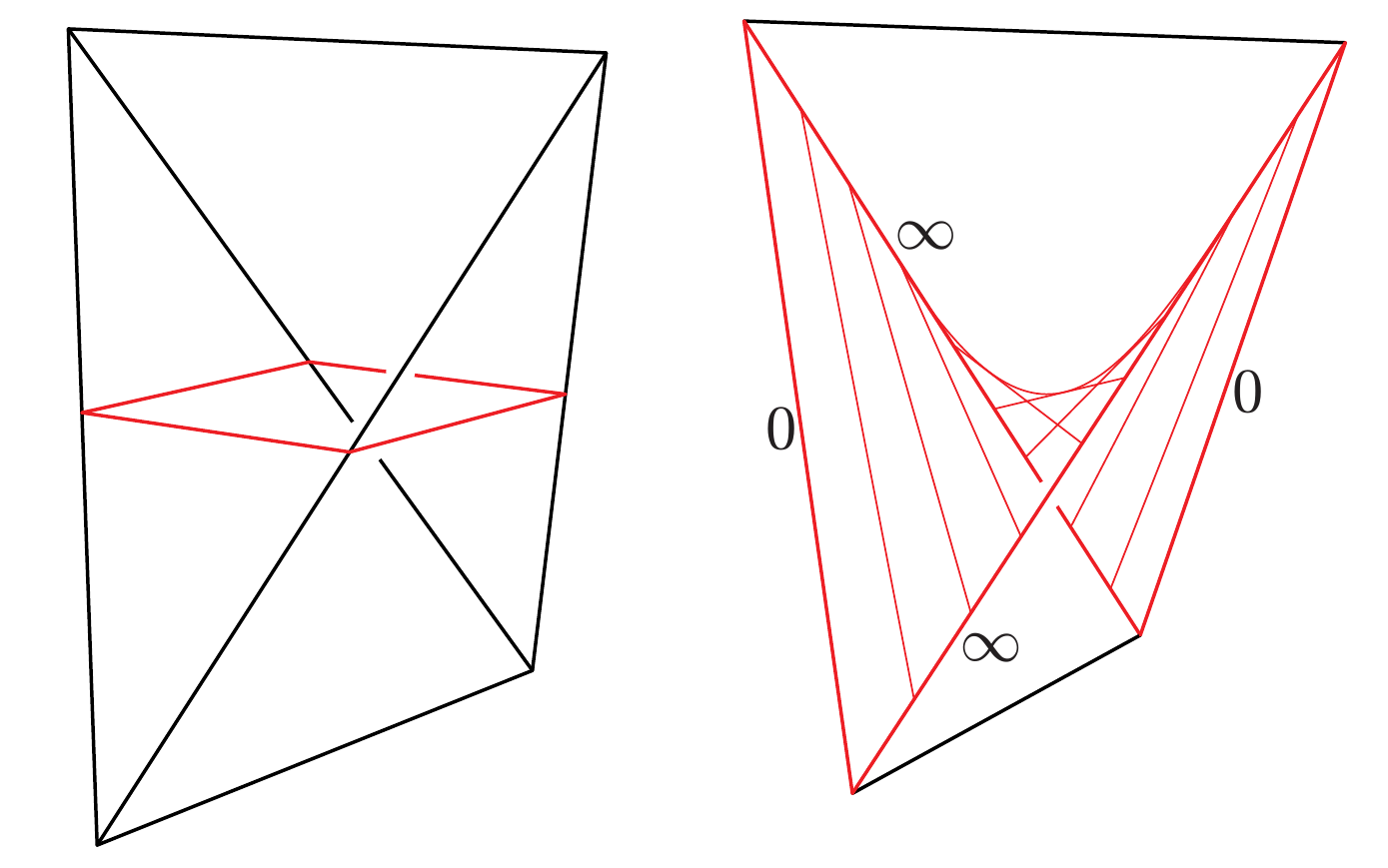}
\caption{A tetrahedron with a quadrilateral and a tetrahedron with the corresponding twisted square, 0 and $\infty$ edges labelled.}
\label{fig_quad_and_tw_sq.pdf}
\end{figure}

The Q-matching equations ensure that the pieces of a spun-normal surface match up around an edge of the tetrahedralisation. These equations correspond to the requirement in the Yoshida construction that the number of 0-edges matches the number of $\infty$-edges of twisted squares that meet at each edge of the tetrahedralisation\footnote{The slopes of corners of quadrilaterals incident at an edge go in either a right hand or left hand screw direction around the edge, the Q-matching equations state that the numbers of these are equal. After the first isotopy (see section \ref{isotopy1}) each right hand screw direction slope becomes a 0-edge of a twisted square, and each left hand screw direction slope becomes an $\infty$-edge.}. The Yoshida construction continues by removing parts of the twisted squares within a small tubular neighbourhood $\mathcal{N}_e$ of each edge $e$, then joining the resulting edges of the twisted squares to each other using long thin strips parallel to $e$ which each link a 0-edge to an $\infty$-edge. See figure \ref{fig_thin_strips.pdf} for an example which shows that there may be choices in this step of the construction.\\

\begin{figure}[htbp]
\centering
\includegraphics[width=1.0\textwidth]{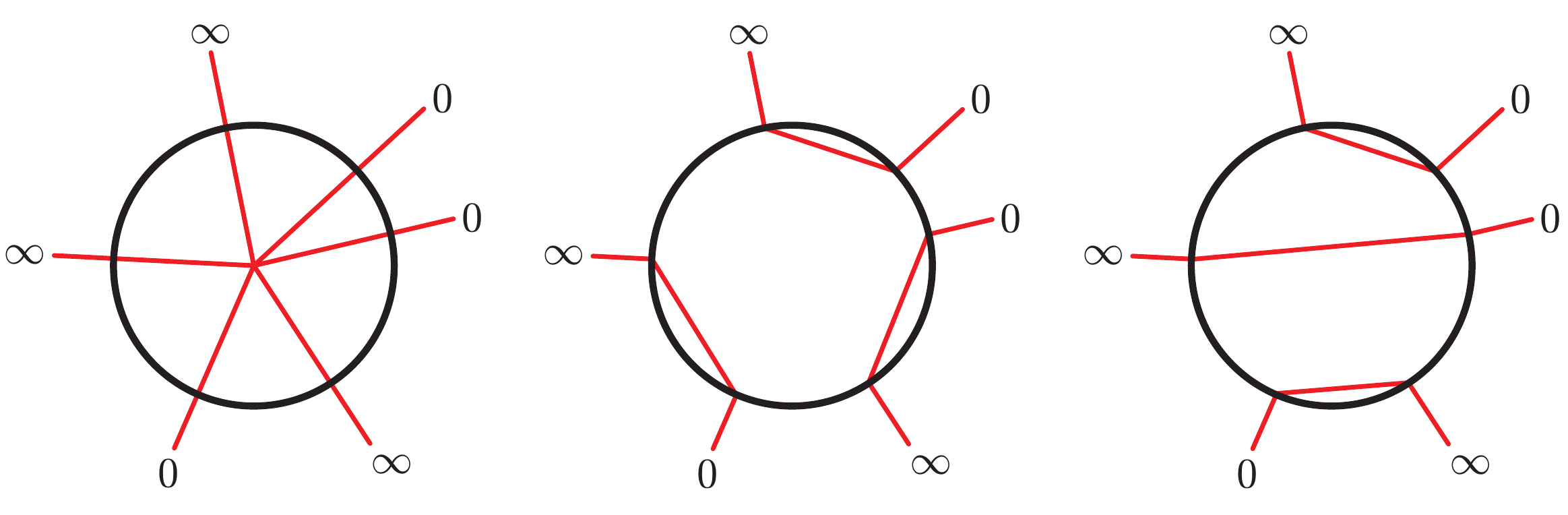}
\caption{Cross-section view of parts of twisted squares within a $\mathcal{N}_e$ before removing them, and two different ways to rejoin 0 and $\infty$ edges of twisted squares to each other with long thin strips.}
\label{fig_thin_strips.pdf}
\end{figure}

The last step of the construction deals with what happens near the cusp(s). First, we truncate all of our tetrahedra, which also cuts off the corners of the twisted squares. The resulting triangulation of the boundary of the manifold contains the boundaries of the twisted squares and boundaries of the thin strips, which form curves. We look for any null-homotopic curves here, and starting with the innermost such curve, we cap each off with a disk and push the disk inside the manifold. This completes the Yoshida construction of a twisted squares surface.

\begin{rmk}\label{double_covers}
Both Tillmann and Yoshida in their papers generating essential surfaces deal with any one-sided surfaces they generate by taking double covers, Tillmann by scaling the admissible rational solution to the Q-matching equations by the minimal integer to result in a two-sided surface, and Yoshida by taking a double cover of any one-sided components of the twisted squares surface generated. We do not assume these steps in the construction of spun-normal or twisted squares surfaces. They also both go on to perform compressions (and boundary compressions for Yoshida) in order to produce incompressible surfaces, and again we do not take these steps. 
\end{rmk}

\section{The Isotopy}

We begin with a surface $S$ in spun-normal form relative to an ideal tetrahedralisation $\mathcal{T}$ of a 3-manifold $M$ with boundary $\bdry M$. Let $\mathcal{E}$ be the edge set of $\mathcal{T}$ and let $\mathcal{N}_e$ be small disjoint tubular neighbourhoods of the edges $e \in \mathcal{E}$. See figure \ref{fig_glossary_picture.pdf}.\\

\begin{figure}[htbp]
\centering
\includegraphics[width=0.8\textwidth]{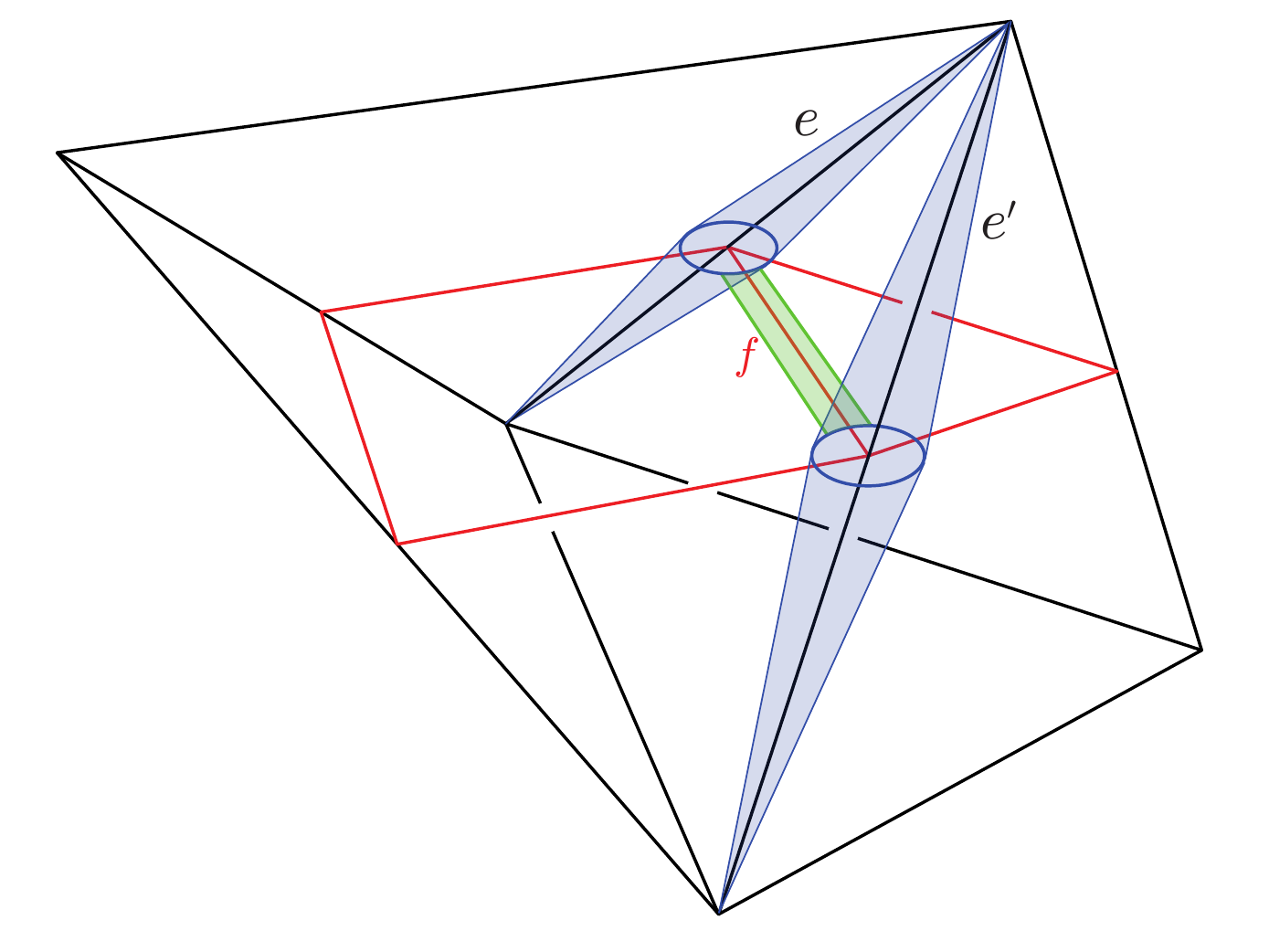}
\caption{Two tetrahedra with a quadrilateral and triangle. Labelled are two edges of $\mathcal{T}$, $e$ and $e'$. $\mathcal{N}_e$ and $\mathcal{N}_{e'}$ are shaded in blue. Also labelled is the edge between the quadrilateral and the triangle, $f$. $\mathfrak{N}_f$ is shaded in green.}
\label{fig_glossary_picture.pdf}
\end{figure}

We will describe an isotopy of $S$ which will convert all of the quadrilaterals in $S$ into twisted squares. The triangles of $S$ will correspond to either parts of disks which cap off null-homotopic boundary curves of the resulting twisted squares, or parts of the surface which spiral around into the cusp.

\subsection{First stage of the isotopy}\label{isotopy1}

\begin{figure}[htbp]
\centering
\includegraphics[width=\textwidth]{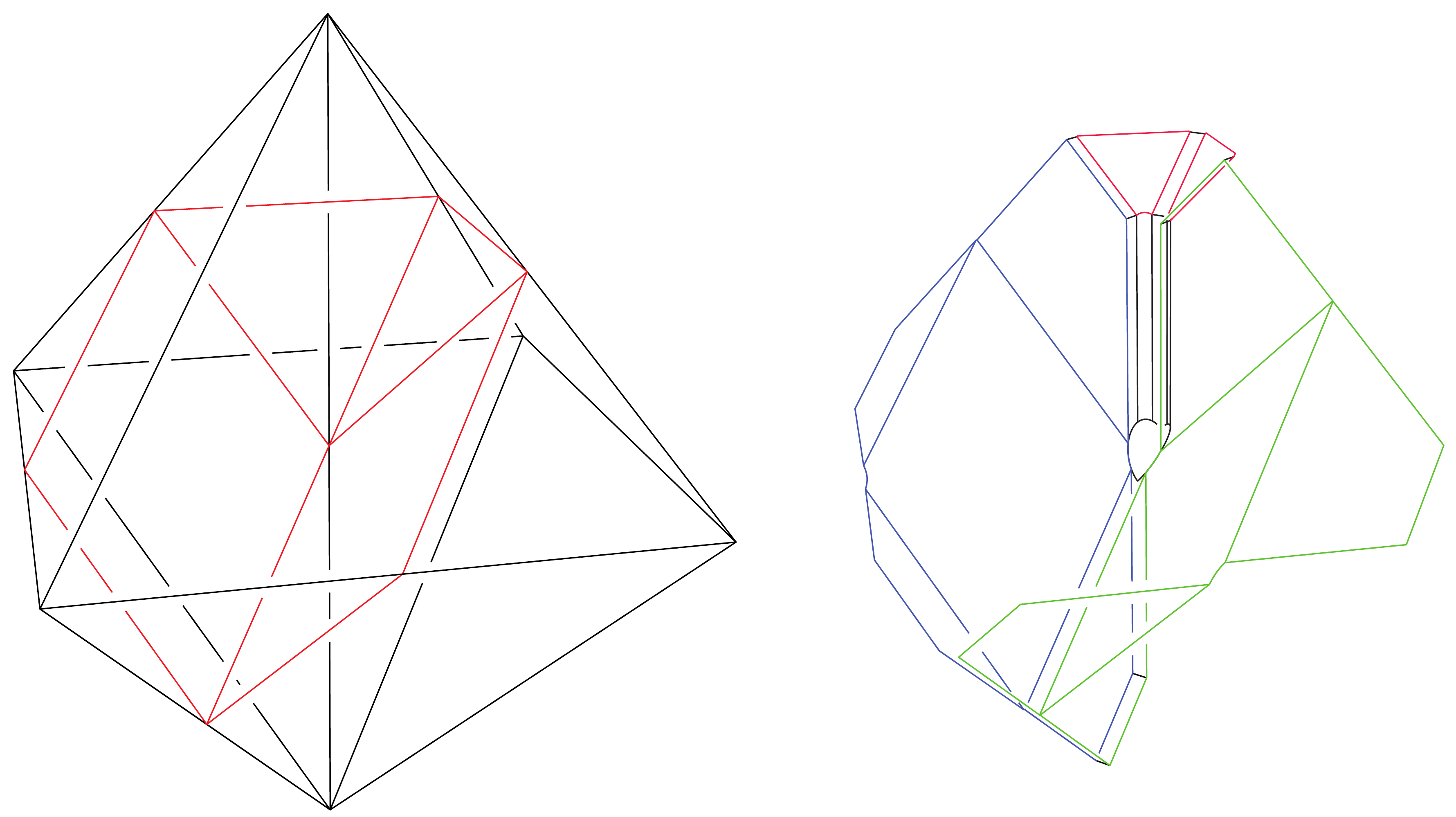}
\caption{Part of a surface in spun-normal form and the first stage of the isotopy.}
\label{fig_spunnorm_qqtt.pdf}
\end{figure}

See figure \ref{fig_spunnorm_qqtt.pdf}. In the left diagram we have four tetrahedra arranged around a central vertical edge, and part of a spun-normal surface, consisting of two quadrilaterals and two triangles. In the right diagram we have the first stage of the isotopy, with the tetrahedra removed for clarity. The two triangles are still shown in red, the two quads have become the saddle shapes shown in blue and green. We describe in general how to perform this first step as follows:\\

First, fix the parts of surface inside all the $\mathcal{N}_e$ (the components are all parallel disks that cut across the tube). Now each edge $f$ in the edge set of the spun-normal surface $S$ is a normal curve in a triangle of $\mathcal{T}$, and so there is one vertex of that triangle between the two ends of $f$. We will pull the surface near $f$ towards that vertex. At the same time we pull any triangle parts of the spun-normal surface towards the vertex they are nearest to.\\

For the $f$ in the edge set of the spun-normal surface let $\mathfrak{N}_f \subset S$ be small disjoint 2-dimensional neighbourhoods of each $f \setminus (\mathcal{N}_e \cup \mathcal{N}_{e'})$ where $e, e' \in E$ are the two edges of $\mathcal{T}$ at either end of $f$. See figure \ref{fig_glossary_picture.pdf}. This is the section of $S$ that we pull towards its vertex, together with the neighbourhoods of any $f'$ edges on parallel sheets of $S$. We also pull all parallel triangle parts of $S$ towards their vertex.  \\

We cannot of course break the surface as we pull parts of it towards the vertices, so as we pull we introduce strips in a small neighbourhood of the boundary between the $\mathfrak{N}_f$ and triangles, and the rest of the surface (the parts inside the $\mathcal{N}_e$ and the rest of the quadrilaterals) in order to maintain the connection between these parts. The quadrilaterals together with the added strips become twisted squares, and we obtain the right hand diagram of figure \ref{fig_spunnorm_qqtt.pdf}.\\ 

\subsection{Second stage of the isotopy}

In the second stage we need to deal with the surface parts inside the $\mathcal{N}_e$ in order that we join the twisted squares we get from the first stage to each other using thin strips parallel to the edges $e$.\\

In the example of figure \ref{fig_spunnorm_qqtt.pdf} we obtain two twisted squares that meet at the vertical edge. In the Yoshida construction, these are to be joined by a thin strip of surface, and in this case there is only one way to do it, and here the parts inside $\mathcal{N}_e$ and strips are isotopic (this is the second stage of the isotopy) to a single long thin strip parallel to the vertical edge and joining the two twisted squares.\\

\begin{figure}[htbp]
\centering
\includegraphics[width=\textwidth]{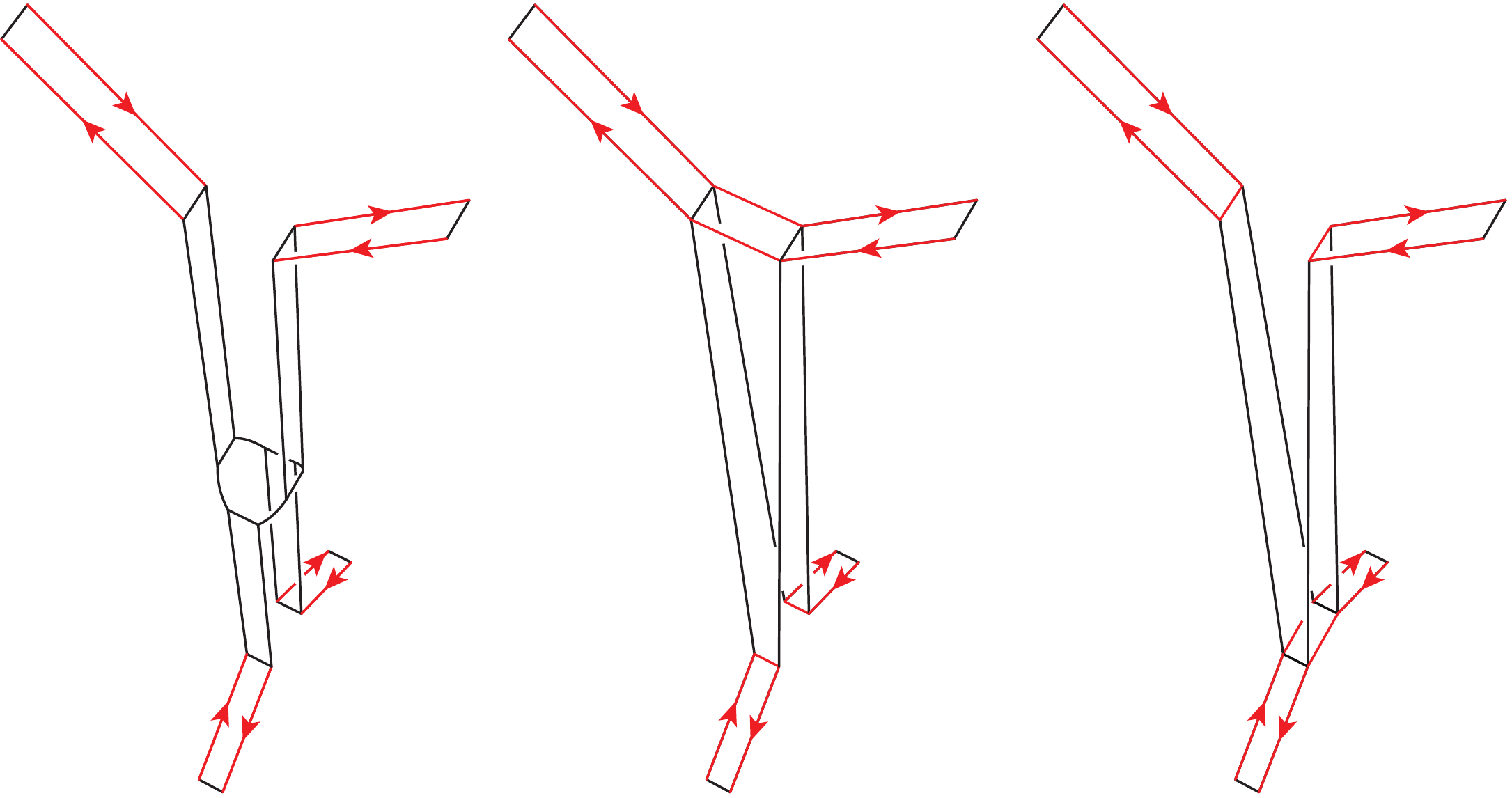}
\caption{Four quadrilateral pieces after the first stage of the isotopy, only the $\mathfrak{N}_f$ and the disk within $\mathcal{N}_e$ are shown. The two ways to convert this into thin strips parallel to $e$. The arrows correspond to the orientations of the boundary of the twisted squares on $\bdry M$ in the Yoshida construction.}
\label{fig_spunnorm_saddle.pdf}
\end{figure}

A more complicated example is shown in figure \ref{fig_spunnorm_saddle.pdf}. Here four quadrilaterals met at the vertical edge, and the left hand diagram shows the picture after the first isotopy. Figure \ref{fig_4_twisted_squares.pdf} shows the twisted squares we get. Here, we get a choice as to how to deal with the disk within the $\mathcal{N}_e$ in the second part of the isotopy: we can either push it up, as in the center diagram, or down, as in the right hand diagram. These correspond to the two ways to join the four edges of the twisted squares with thin strips parallel to the vertical edge.\\

\begin{figure}[htbp]
\centering
\includegraphics[width=0.5\textwidth]{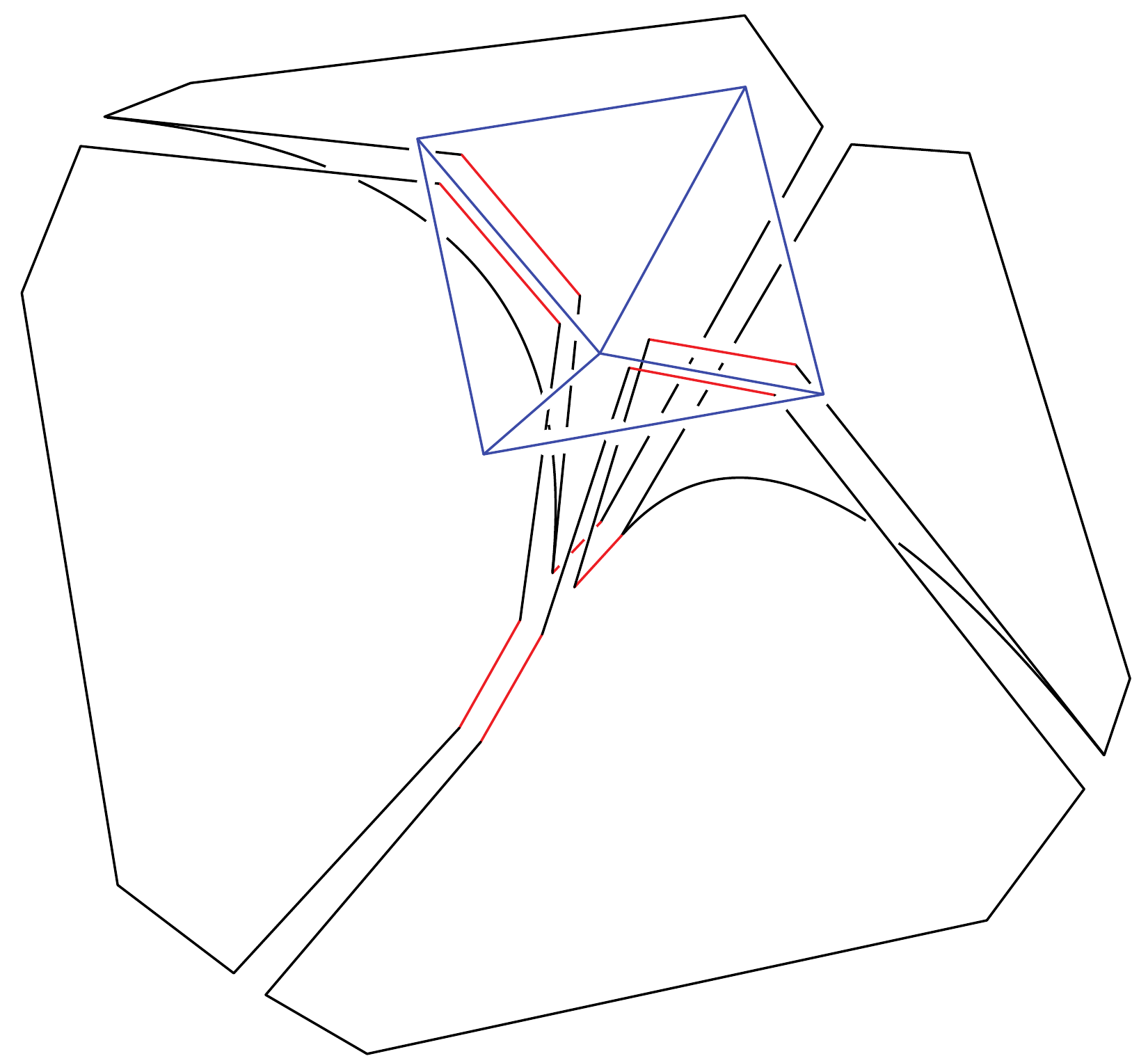}
\caption{Four twisted squares. The truncated ends of the tetrahedra at the top vertex are shown.}
\label{fig_4_twisted_squares.pdf}
\end{figure}

In general we will have some number of parallel disks in each $\mathcal{N}_e$, each with strips attached to the edge making the disk into a saddle of some order. The second stage of the isotopy consists of pushing the various disk to one end of the edge or the other (in fact some disks can be split apart into two or more regions, each of which can go to either end of the edge). We will investigate this in the next section.

\section{Choices in the Yoshida construction}

\begin{figure}[htbp]
\centering
\includegraphics[width=1.0\textwidth]{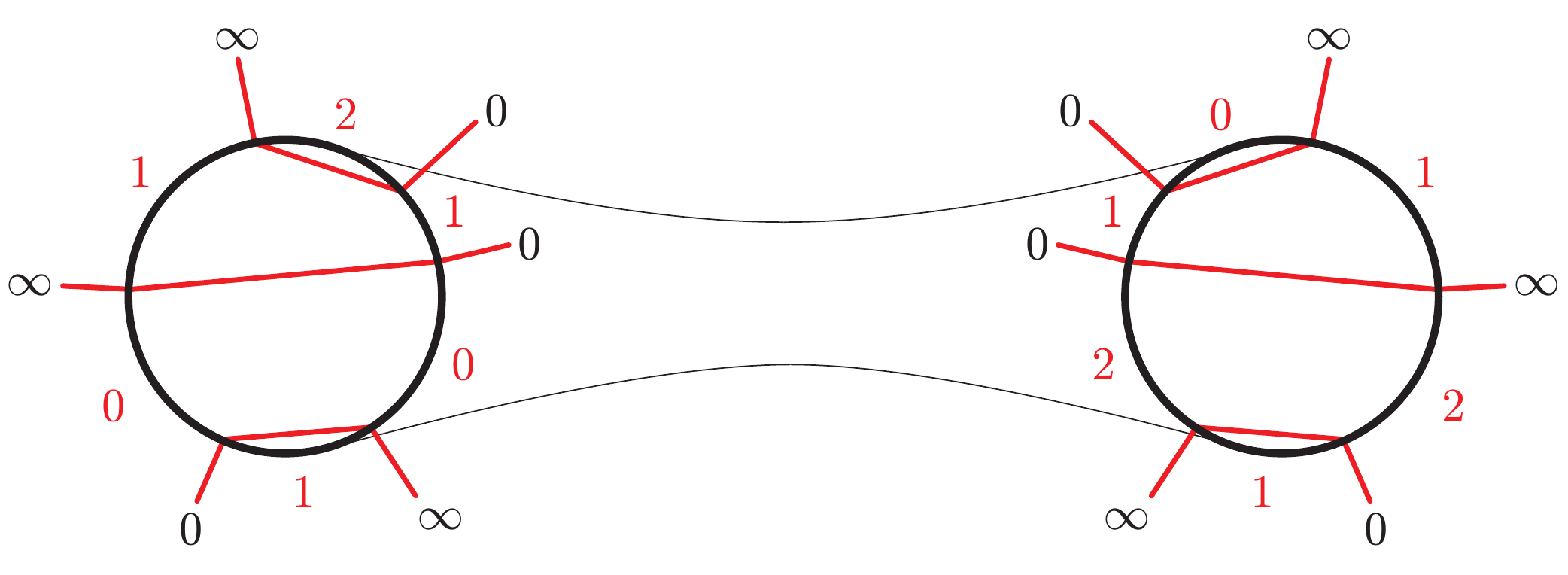}
\caption{Two ends of a $\mathcal{N}_e$ with a choice of long thin strips.}
\label{fig_thin_strips_push_saddles.pdf}
\end{figure}

The choices made in constructing a given Yoshida form surface are of which edges of twisted squares to join with each other. Given a choice of those joins we will decide how to push each disk of the surface after the first stage of the isotopy in order to realise those joins. Consider figure \ref{fig_thin_strips_push_saddles.pdf}, which shows both ends of a $\mathcal{N}_e$ with a choice of long thin strips. Note of course that the picture at one end is a mirror reflection of the other. We have added (in red) integers around each end, with an integer in between each sheet of the twisted squares meeting at $e$. Moving anti-clockwise around the end, the red integer increases by one when we pass a 0 incoming edge, and decreases by one for each $\infty$ edge. Because the number of 0 edges is equal to the number of $\infty$ edges, these integers can be chosen consistently, and if we require that the smallest integer is 0 then they are uniquely defined.\\

If the maximum number we reach is $n$, then the corresponding number at the other end of $e$ to an integer $k$ is $n-k$. Notice that $n$ is also the number of disks that sit in the center of $e$ after the first step of the isotopy from spun-normal form. Observe also that the regions of the disk at an end of $\mathcal{N}_e$ is cut into a number of regions by the ends of the long thin strips, and the integers on the boundary of a region are all the same, so we associate that integer with the whole region. This information tells us how to push the disks from the center of $e$ out to the ends: we deal with each region of the disks independently. For a region of the end with associated integer $k$, we push $k$ sheets to this end and $n-k$ sheets to the other. The above observations show that we can consistently do this, and we end up with the desired long thin strips.\\

We now find ourselves at the end of the second isotopy from spun-normal form: we have isotoped so that all quadrilaterals are now twisted squares, joined to each other through the desired choice of long thin strips. 

\begin{thm}\label{same away from bdry}
Given a twisted squares surface $T$ and the spun-normal surface $S$ produced from the same data, $T$ and $S$ may be isotoped so that outside of a small neighbourhood of $\bdry M$ in $M$ the two surfaces coincide.
\end{thm}

\begin{proof}
The above isotopy of $S$ produces a surface $S'$. Now delete from $S'$ all $\mathfrak{N}_f$ and triangle parts.  We obtain a surface $S''$ which has boundary in the interior of the manifold $M$, but whose boundary is all close to $\bdry M$. Choose a suitable product neighbourhood $N_S$ of $\bdry M$ that contains the boundary of $S''$. Choose a suitable product neighbourhood $N_T$ of $\bdry M$ which contains any capped off boundary curves of $T$. Then $S'' \setminus (N_S \cup N_T)$ is identical to $T \setminus (N_S \cup N_T)$.
\end{proof}

\section{Compressions}

We now need to understand the parts of the surface parallel to the boundary, which is made out of parts that were triangles in spun-normal form and the strips $\mathfrak{N}_f$. For simplicity we restrict from now on to manifolds $M$ such that $\bdry M$ is a union of tori. We consider the curves $C$ forming the boundary of $S''$ from the proof of theorem \ref{same away from bdry}. Let $N$ be a closed product neighbourhood of $\bdry M$ such that $\bdry N$ contains all of the curves on the boundary of $S''$. We may put $S'$ in the same picture in such a way that all $\mathfrak{N}_f$ and triangle parts are contained in $N$ and their intersection with $\bdry N$ consists only of $C$. The key observation is that the pattern of the curves $C$ against the triangulation of $\bdry N$ determines how and if those curves are connected together within $N$ by the $\mathfrak{N}_f$ and triangles.\\

\begin{figure}[htbp]
\centering
\includegraphics[width=0.4\textwidth]{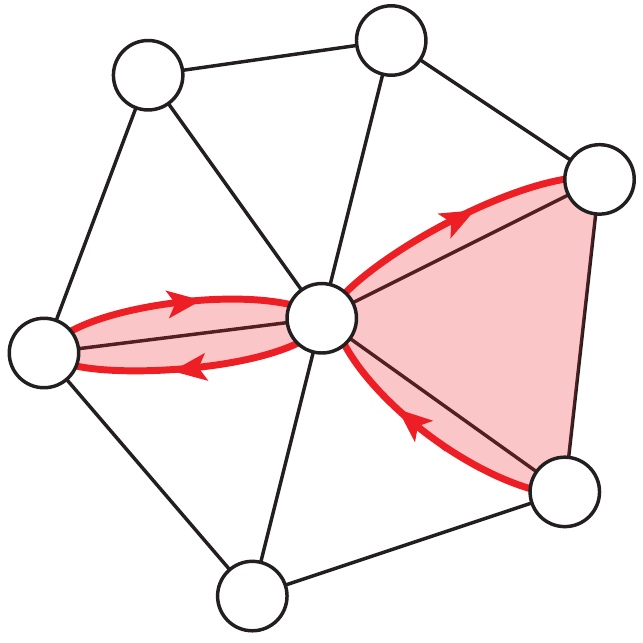}
\caption{Parts of $C$ on $\bdry N$.}
\label{fig_tris_and_strips.pdf}
\end{figure}

In figure \ref{fig_tris_and_strips.pdf} we see an example of a small part of the triangulation on $\bdry N$. The curves $C$ consist of the boundaries of twisted squares (which were once quadrilateral pieces), shown as red arcs here in the triangular end of the tetrahedron containing the twisted square, alternating with the ends of thin strips within $\mathcal{N}_e$ which are not shown here. The strips $\mathfrak{N}_f$ join onto these arcs crossing over the edges of the triangulation. At the other side of the $\mathfrak{N}_f$ can be either another arc boundary of a twisted square (as on the left of the figure) or a triangle (as on the right). Also in figure \ref{fig_tris_and_strips.pdf} we have oriented the arcs in such a way that the arrow points anti-clockwise around the triangle within which the arc is. With this orientation we can also see that the $\mathfrak{N}_f$ and triangle parts meet the arc coming from the right of the arrows. Note that the orientation here is the same as the orientation of the boundary of the twisted squares on $\bdry M$ in the Yoshida construction, and that when we link the arcs at the $\mathcal{N}_e$ the orientation of the arcs linked agree. Beware however that this orientation is not induced by an orientation of the surface.

\begin{defn}
A {\bf boundary annulus} for a surface $F$ in a 3-manifold $M$ with non-empty boundary $\bdry M$ is an annulus $A$ embedded in $M$ with boundary components $\alpha \subset F$ and $\beta \subset \bdry M$ and such that $A \cap F = \alpha$ and $A \cap \bdry M = \beta$.
\end{defn}
We also define a compression along a boundary annulus analogously to compression along a disk. That is, we cut the surface along $\alpha$, where it intersects the annulus and stitch in two parallel copies of the annulus. This move introduces two boundary components to the surface. See figure \ref{fig_boundary_annulus_compression.pdf}.

\begin{figure}[htbp]
\centering
\includegraphics[width=0.9\textwidth]{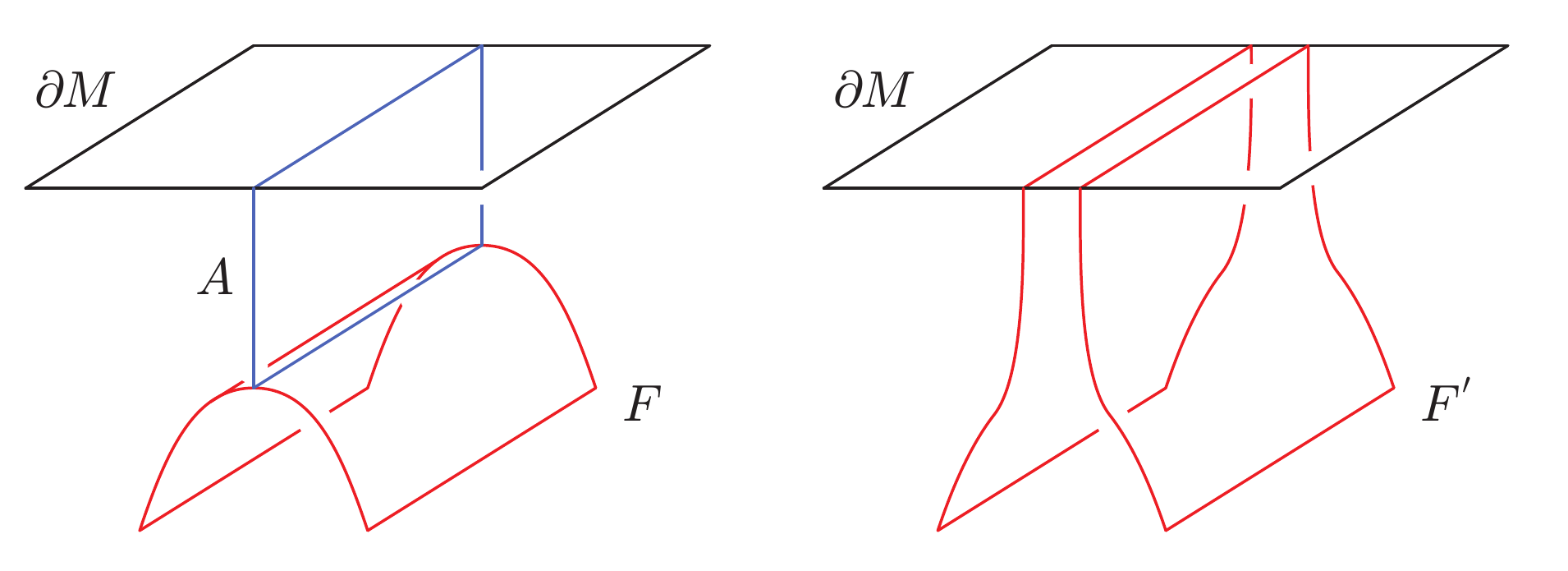}
\caption{A boundary annulus $A$ and the result of compression along it.}
\label{fig_boundary_annulus_compression.pdf}
\end{figure}

\begin{thm}
\label{get_any_yoshida}
For a 3-manifold $M$ with $\bdry M$ a union of tori, if $T$ is a twisted squares surface and $S$ the spun-normal surface produced by the same data then $T$ can be obtained from $S$ by the correct choices made of isotopy and a finite number of compressions along disks and boundary annuli, and removing boundary parallel components. \end{thm}
\begin{proof}
We start by isotoping $S$ and orienting the curves $C$ as in the discussion at the start of this section. Each curve on each torus falls into one of 3 categories, and we deal with each of them in order:

\begin{enumerate}
\item Null-homotopic curves with arrows pointing in the clockwise direction.
\item Null-homotopic curves with arrows pointing in the anti-clockwise direction.
\item Some number of parallel non null-homotopic curves.
\end{enumerate}
In case 1, the $\mathfrak{N}_f$ and triangles connected to such a curve cap off the curve in the same way as in the twisted square construction and so no compression move needs to be done.\\

In case 2 the $\mathfrak{N}_f$ and triangles connected to the curve go outwards from the curve rather than inwards. We perform a compression move on the disk bounded by this curve (innermost first if necessary) which caps off the curve on the inside, as in $T$. It also caps off the "outside" surface, removing a funnel going down into the interior of the manifold. \\

We are left with only the curves in case 3. There are two possibilities for what the part of surface that meets one of these curves now looks like. Either it is unbounded and spins infinitely many times around the torus boundary towards the cusp, or it is bounded, and is an annulus joining this curve to another parallel non null-homotopic curve. Note that the orientations on these two curves must be opposite for this possibility to occur.\\

In the first case we can "unwind" the spinning with an isotopy which leaves the surface in the same state as if we had simply deleted all of the triangles and $\mathfrak{N}_f$ strips and replaced them with an annulus connecting the curve on $\bdry N$ straight out to $\bdry M$.\\

In the second case, in order to recover $T$ we need to cut the annulus and move the resulting edges out to $\bdry M$. We can do this by doing a compression (outermost first if necessary) along a boundary annulus with one edge on the core of the annulus we are cutting and the other on $\bdry M$.\\

Finally, it is possible if there were curves in case 2 above but none in case 3 that we may have some boundary parallel tori parts of our surface left over, which we delete. The resulting surface is $T$.\\
\end{proof}

\section{Actions on trees}

For this section, we use the phrase "$F$ is dual to an action of $\pi_1M$ on a tree $\mathfrak{T}$" in the same sense as Shalen~\cite{handbook_shalen}, section 2.2. That is, we can construct the surface $F$ in the following way: We have a tree $\mathfrak{T}$, $E$ the set of midpoints of edges of $\mathfrak{T}$ and a map $\til{f}:\til{M}\rightarrow \mathfrak{T}$ which is $\pi_1M$ equivariant and transverse to $E$, and $\til{f}^{-1}(E) = p^{-1}(F)$, where $p$ is the covering projection from $\til{M}$ to $M$.

\begin{thm}\label{dual_same}
If $M$ is an oriented 3-manifold with $\bdry M$ a union of tori and $S$ is a two-sided\footnote{In the context of \cite{handbook_shalen} we work with $\pi_1M$ acting on $\mathfrak{T}$ without inversions, and by 2.3.3 of \cite{handbook_shalen} any such dual surfaces must be two-sided.} spun-normal surface dual to a given action of $\pi_1M$ on a tree $\mathfrak{T}$ and $T$ is a twisted squares surface produced from the same data then $T$ is also dual to the action. 
\end{thm}

\begin{proof}
In section 2.4 of \cite{handbook_shalen}, Shalen shows how to modify a surface dual to an action by a (disk) compression to produce a new surface which is also dual to the action. He also shows how to remove boundary parallel components. We need to show that the other modification (boundary annulus compression) results in surfaces dual to the action.\\

If $F$ is a surface dual to an action of $\pi_1M$ on $\mathfrak{T}$ with a boundary annulus $A$ as in figure \ref{fig_boundary_annulus_compression.pdf}, we can modify $\til{f}$ as follows: Let $N$ be a small tubular neighbourhood of $A$ in $M$ (see figure \ref{fig_boundary_annulus_nbd.pdf}). Lift $F$, $A$ and $N$ to $\til{F}$, $\til{A}$ and $\til{N}$ in $\til{M}$. Now let $\til{f}'$ be identical to $\til{f}$ outside of $\til{N}$ and be related by a $\pi_1M$ equivariant homotopy to $\til{f}$ inside $\til{N}$ which pushes $\til{F}$ out to the boundary, stretching the parts of $\til{N} \setminus \til{F}$ that do not touch $\bdry \til{M}$ until it does touch $\bdry \til{M}$. The resulting $\til{f}'$ is $\pi_1M$ equivariant and with a little care can be made to be transverse to $E$. The resulting dual surface $F'$ is obtained from $F$ by a boundary annulus compression.
\end{proof}

\begin{figure}[htbp]
\centering
\includegraphics[width=0.6\textwidth]{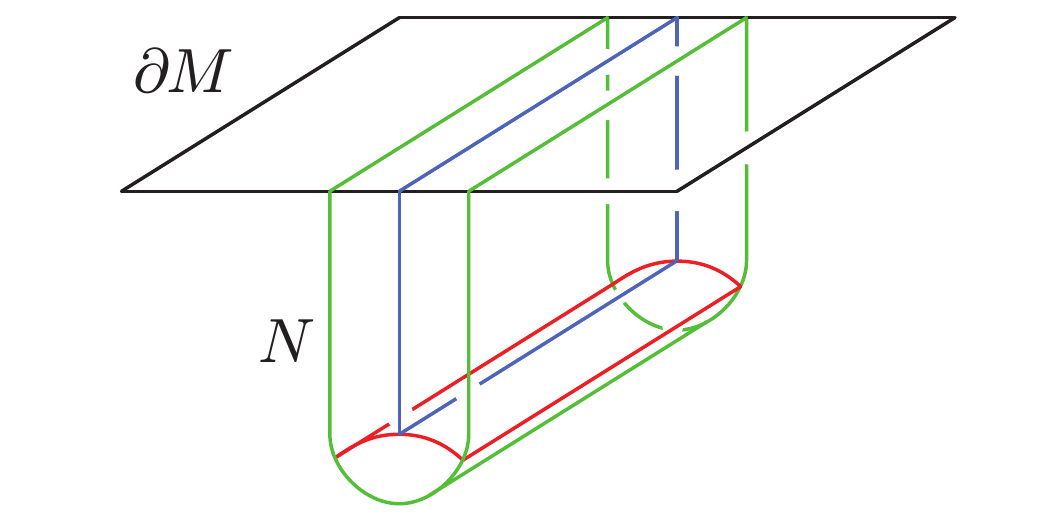}
\caption{A small tubular neighbourhood of a boundary annulus, the part of the surface $F$ outside of $N$ is not shown.}
\label{fig_boundary_annulus_nbd.pdf}
\end{figure}

\begin{thm}\label{twisted_squares_detected}
Let $M$ be an oriented 3-manifold with $\bdry M$ a union of tori with ideal triangulation $\mathcal{T}$ and $T$ a two-sided twisted squares surface obtained via Yoshida's construction from an ideal point of the deformation variety $\mathfrak{D}(M, \mathcal{T})$ which corresponds to an ideal point of the character variety. Then any essential surface obtained from $T$ by compressions is detected by the character variety.
\end{thm}

Here correspondence between an ideal point of the deformation and character varieties means that as we approach the ideal point in the deformation variety, the corresponding characters approach an ideal point of the character variety. See \cite{tillmann_degenerations} section 2.5 for more detail on the map between the varieties. We say that a surface is detected by the character variety if it can be produced by the Culler-Shalen construction from an ideal point of the character variety (see \cite{handbook_shalen}). As in remark \ref{double_covers}, we do not follow Yoshida~\cite{yoshida91} in taking double covers of only components that are not two-sided, but rather double up all components, or equivalently only consider integer solutions that result in two-sided twisted squares surfaces.  

\begin{proof}[Proof of theorem \ref{twisted_squares_detected}]
First assume that the spun-normal surface $S$ obtained from the same data as we use to build $T$ is two-sided. Proposition 23 of Tillmann \cite{tillmann_degenerations} gives us (although not in this language) that the spun-normal surface corresponding to an ideal point of the deformation variety that corresponds to an ideal point of the character variety is dual to the action on a tree associated to the ideal point of the character variety (we need that $S$ is two-sided and hence bicollared for this). The data used to build $T$ from the ideal point is an admissible integer solution to the Q-matching equations, which is the same as for Yoshida's construction, so $T$ and $S$ are related as in theorem \ref{dual_same}. Therefore $T$ is also dual to the ideal point of the character variety, and so is any essential surface obtained from $T$ by compressions.\\

If we are in the case that $T$ is two-sided but $S$ is not, then doubling all of the numbers of quadrilaterals in the construction of the spun-normal surface results in the double cover of $S$ which is two-sided, and so the above argument goes through. The result gives us that the double cover of $T$ (which is what we get by doubling all of the numbers of twisted squares) is dual to the ideal point of the character variety. Since $T$ was two-sided already, this surface consists of two parallel components for each component of $T$. We follow the Culler-Shalen construction through with two parallel copies until at the very end, at which point we may take only one parallel component, and again we produce the essential surface obtained from $T$ by compressions.
\end{proof}
\begin{rmk}
It isn't clear if the two constructions (spun-normal and twisted squares surfaces) from an ideal point of the deformation variety can always produce the same essential surfaces, but the above results (and primarily Tillmann's result) do show that in the case that there is a corresponding ideal point of the character variety, all of the essential surfaces we get from either construction can also be obtained by the Culler-Shalen construction. However,  explicitly carrying out the algorithms seem to be more feasible for the two deformation variety constructions.
\end{rmk}

\section{Incompressible surfaces in general}
We can also draw a connection with incompressible surfaces in general (as opposed to specifically those produced by some degeneration of the geometric structure) using the following theorem:
\begin{thm}[{Theorem 1.5 of Walsh \cite{walsh_incomp}}]
Let $M$ be a atoroidal, acylindrical, irreducible, compact three-manifold with torus boundary components, and $\mathcal{T}$ an ideal triangulation of $M$ with essential edges. Let $S$ be a properly embedded, two-sided, incompressible surface in $M$ that is not a virtual fiber. Then $S$ can be isotoped to be in normal or spunnormal form in $(M, \mathcal{T})$.
\end{thm}
\begin{cor}
With the above notation and conditions, there exists a twisted squares surface $T$ embedded in $M$ such that $T$ can be obtained from $S$ by a number of boundary annulus compressions.
\end{cor}
\begin{proof}
$S$ is incompressible, so there cannot be any disk compression moves to make, only boundary annulus compressions.
\end{proof}

\section{Further questions}
\begin{enumerate}
\item Are there any manifolds with essential surfaces that can be produced by compressions from only one of the spun-normal and twisted squares constructions and not the other? The difference seems to come down to the boundary annulus compressions.
\item Are there any manifolds with essential surfaces that can be produced only via the Culler-Shalen construction and not by either the twisted squares or spun-normal constructions? If so, does there exist some other way to construct surfaces using the deformation variety that does obtain all of the essential surfaces that the Culler-Shalen construction does?
\item Tillmann~\cite{tillmann_degenerations} also asks the reverse question to the above, whether there are closed essential surfaces that are not detected by the character variety but that are detected by the deformation variety (i.e. produced by the spun-normal construction in the case that the ideal point of the deformation variety does not correspond to an ideal point of the character variety)? We can expand the question to ask the same for twisted squares surfaces.
\item Is there an intrinsic interpretation (not depending on the triangulation of the manifold) of the orientation of curves formed from the boundaries of twisted squares on $\bdry M$ in a twisted squares surface? There are parallel curves with opposite orientations if and only if there are boundary annulus compressions to make in altering the spun-normal surface to the twisted squares form, and so understanding these orientations could help us answer the first question in this list.
\end{enumerate}
\bibliographystyle{hamsplain}
\bibliography{/Users/h/Math/research/henrybib}
\end{document}